\documentclass[12pt]{amsart}

\usepackage{amsmath,amscd,amsfonts,amsthm,amssymb,euscript}
\usepackage{fullpage}

\def\stacksum#1#2{{\stackrel{{\scriptstyle #1}}
{{\scriptstyle #2}}}}
\newcommand{\ideal}[1]{\mathfrak{{#1}}}
\newcommand{\mods}[1]{\,(\mathrm{mod}\,{#1})}
\newcommand{\pit}[1]{\beta_{{#1}}}
\renewcommand{\leq}{\leqslant}
\renewcommand{\geq}{\geqslant}

\DeclareMathOperator{\SL}{SL}

\DeclareMathOperator{\GL}{GL}

\DeclareMathOperator{\End}{End}

\DeclareMathOperator{\Gal}{Gal}

\DeclareMathOperator{\Aut}{Aut}
\DeclareMathOperator{\Sp}{Sp}

\newcommand{\field}[1]{\mathbb{#1}}

\newcommand{\Q}{\field{Q}}

\newcommand{\Z}{\field{Z}}
\newcommand{\F}{\field{F}}

\newcommand{\C}{\field{C}}
\renewcommand{\P}{\field{P}}

\newcommand{\ra}{\to}

\newcommand{\MM}{\mathcal{M}}

\newcommand{\set}[1]{\{#1\}}

\newcommand{\beq}{\begin{displaymath}}
\newcommand{\eeq}{\end{displaymath}}
\newcommand{\beqn}{\begin{equation}}
\newcommand{\eeqn}{\end{equation}}

\newcommand{\charac}{\mathrm{char}}
\newcommand{\Fqbar}{\overline{\F}_q}

\theoremstyle{plain}
\newtheorem{thm}{Theorem}
\newtheorem{prop}[thm]{Proposition}
\newtheorem{cor}[thm]{Corollary}
\newtheorem{lem}[thm]{Lemma}
\newtheorem*{intro}{Theorem}

\theoremstyle{definition}

\newtheorem{exmp}[thm]{Example}
\newtheorem{que}[thm]{Question}

\theoremstyle{remark}
\newtheorem{rem}[thm]{Remark}

\begin{document}
\title[Non-simple abelian varieties]{Non-simple abelian varieties in a
  family: geometric and analytic approaches}
\author[J. Ellenberg]{Jordan S. Ellenberg}
\address{Department of Mathematics \\ University of Wisconsin \\ 480 Lincoln Drive \\
  Madison, WI 53705 USA} \email{ellenber@math.wisc.edu}
\author[C. Elsholtz]{Christian Elsholtz}
\address{Department of Mathematics\\ Royal Holloway\\
  University of London\\ Egham\\ TW20 0EX Surrey\\ UK}
\email{christian.elsholtz@rhul.ac.uk} \author[C. Hall]{Chris Hall}
\address{Department of Mathematics, University of Michigan at Ann
  Arbor\\ Michigan, USA} \email{hallcj@umich.edu}
\author[E. Kowalski]{Emmanuel Kowalski}
\address{ETH Z\"urich -- D-MATH\\
  R\"amistrasse 101\\
  8092 Z\"urich\\
  Switzerland} \email{kowalski@math.ethz.ch}


\begin{abstract}
  Let $A_t$ be a family of abelian varieties over a number field $k$
  parametrized by a rational coordinate $t$, and suppose the generic
  fiber of $A_t$ is geometrically simple.  For example, we may take
  $A_t$ to be the Jacobian of the hyperelliptic curve $y^2 =
  f(x)(x-t)$ for some polynomial $f$.  We give two upper bounds for
  the number of $t \in k$ of height at most $B$ such that the fiber
  $A_t$ is geometrically non-simple.  One bound comes from arithmetic
  geometry, and shows that there are only {\em finitely} many such
  $t$; but one has very little control over how this finite number
  varies as $f$ changes.  Another bound, from analytic number theory,
  shows that the number of geometrically non-simple fibers grows quite
  slowly with $B$; this bound, by contrast with the arithmetic one, is
  effective, and is uniform in the coefficients of $f$.  We hope that
  the paper, besides proving the particular theorems we address, will
  serve as a good example of the strengths and weaknesses of the two
  complementary approaches.
\end{abstract}

\maketitle

\section*{Introduction}

Given an algebraic family $\set{A_t}_{t \in \Q}$ of abelian varieties
parametrized by a rational number $t$ whose generic fiber has a
certain property, it is natural to ask what one can say about the set
of $t \in \Q$ such that $A_t$ has the same property.  One expects that in many cases
this set will be ``large'' in some sense, even if the property in
question is not a straightforward ``algebraic'' condition.  
\par
We consider in this context the property of geometric simplicity,
which can be approached from several directions. In fact, in some
sense, the main goal of this paper is to use this example to
illustrate and compare different approaches, via arithmetic geometry
and via analytic number theory.  It turns out that neither is clearly
preferable to the other, each method showing characteristic strengths
and weaknesses, which we will try to emphasize. In this spirit, and
for the sake of clarity, we do not always pursue the strongest
possible conclusions.

More precisely, we will discuss the following two theorems, each of
which is a special case of a more general theorem proved in the main
body of the paper. Both concern the family of Jacobians $A_t$ of
hyperelliptic curves defined by affine equations
$$
y^2=f(x)(x-t)
$$
for some squarefree polynomial $f\in \Z[X]$ of degree $2g$, $g\geq
1$. For $t\in \Q$ written $t=a/b$ with coprime integers $a$ and $b$,
let $H(t)=\max(|a|,|b|)$ be the height of $t$. Let then $S(B)$ denote
the set of $t\in \Q$ with $H(t)\leq B$ such that $A_t$ is \emph{not}
geometrically simple.

 \begin{intro}[Arithmetic geometry]
   There exists a constant $C(f)$, depending on $f$, such that
\begin{equation}\label{eq:bound-geometric}
|S(B)|\leq C(f)
\end{equation}
for all $B\geq 1$. In other words, there are only finitely many $t$
for which $A_t$ is not geometrically simple.
 \end{intro}
 
 This is a special case of Theorem~\ref{th:main} in
 Section~\ref{sec:geoI} and is elaborated on in Example~\ref{ex:main} in
 Section~\ref{sec:geoII}.
 
\begin{intro}[Analytic number theory]
  There exist absolute constants $C\geq 0$ and $D\geq 1$ such that we
  have
\begin{equation}\label{eq:bound-analytic}
|S(B)| \leq C(g^2D(\log B))^{11g^2}
\end{equation}
for all $B\geq 1$.
\end{intro}

This is a special case of Theorem~\ref{th:analytic} in
Section~\ref{sec:analytic}, where we have simplified the bound by
worsening it somewhat. (For readers interested in this analytic
approach but who are not familiar with abelian varieties, we have
summarized enough information to understand the basic problem in an
Appendix, which they may want to read now before starting Section~\ref{sec:analytic}).
\par
\medskip
\par
The first theorem may initially appear much stronger.  But note that
in~(\ref{eq:bound-geometric}), we have no idea about the actual value
of $C(f)$, in particular about how it may vary with $f$, whereas in
the second theorem, the bound~(\ref{eq:bound-analytic}) is
\emph{effective} in terms of $f$. In particular this means we can
deduce bounds for similar problems involving families with more than
one parameter, e.g., for Jacobians of
$$
y^2=f(x)(x-t)(x-v),
$$
for fixed square-free $f$ of degree $2g-1$ and parameters $t$, $v\in
k$.  One can also deduce from~(\ref{eq:bound-analytic}) some upper
bound for the smallest height of a $t$ such that $A_t$ is
geometrically simple, namely there exists some $t$ of height $\leq B$
for which $A_t$ is geometrically simple, where
$$
B=C'(D'g^4)^{11g^2}
$$
for some constants $C'>0$, $D'\geq 1$ (computable in terms of $C$ and
$D$).
\par
The situation may be compared with the problem of counting rational
points on a plane curve $X$ of genus $\geq 2$. The theorem of Faltings
shows that this set of points is finite, but it gives no effective
bound for the heights of the solutions, and only estimates depending
badly on $X$ for the number of points. On the other hand, the method
of Heath-Brown in \cite{heat:2d} yields a completely explicit bound,
depending only on the degree of $X$, for the number of points on $X$
of height at most $B$.
\par
\medskip
\par
Going further with the analogy, we may notice that Caporaso, Harris,
and Mazur~\cite{capo:chm} have shown that if a certain conjecture of
Lang~\cite{lang} holds, then there is a bound \emph{depending only on
  $g$} for the number of rational points on a curve of genus $g$ over
$\Q$.  This suggests the following rather speculative question about
the topic of the current paper:

\begin{que}
Is there an absolute constant $C$ such that, for any
squarefree polynomial $f \in \Z[x]$, there are at most $C$ rational
numbers $t$ such that the Jacobian of $y^2 = f(x)(x-t)$ is
geometrically non-simple? 
\end{que}

If the question is relaxed to allow $C$ to depend on the degree of $f$
(i.e., the genus of the hyperelliptic curves under consideration),
then Lang's conjecture implies an affirmative answer: as we shall see,
the proof of Theorem~\ref{th:main} is based on showing that $S(B)$
maps injectively to the set of rational points on one of a finite set
of curves of sufficiently large genus, where the number and genera of
these curves are bounded in terms of $\deg(f)$.

One can be even more ambitious and ask the following purely geometric
question:

\begin{que}
  Is there an absolute constant $C$ such that, for any squarefree
  polynomial $f \in \C[x]$ of degree at least $6$, there are at most
  $C$ complex numbers $t$ such that the Jacobian of $y^2 = f(x)(x-t)$
  is not simple?
\end{que}

Geometrically, we are asking whether there is an absolute bound on the
number of complex intersection points between certain rational curves
in $\MM_g$ and the sublocus of $\MM_g$ parametrizing curves whose
Jacobians are non-simple.  The difficulty arises from the fact that
the non-simple locus is a countable union of proper subvarieties, so
it is certainly not obvious a priori that there are finitely many $t
\in \C$ for which $A_t$ is non-simple.  Indeed, when $g=2$, the
non-simple locus is a countable union of divisors, so a typical curve
intersects this locus infinitely many times; this is the reason we
require $\deg(f) \geq 6$.
\par
\medskip
\par
\textbf{Acknowledgments.} We wish to thank F. Voloch for many helpful
conversations. The first-named author's work was partially supported
by NSF-CAREER Grant DMS-0448750 and a Sloan Research Fellowship.
\par
\medskip
\par

\textbf{Notation.} As usual, $|X|$ denotes the cardinality of a set,
and $\F_q$ is a field with $q$ elements. For a number field $k$,
$\Z_k$ denotes its ring of integers, and for a prime ideal
$\ideal{p}\subset \Z_k$, $\F_{\ideal{p}}$ is the residue field
$\Z_k/\ideal{p}$.
\par
By $f\ll g$ for $x\in X$, or $f=O(g)$ for $x\in X$, where $X$ is an
arbitrary set on which $f$ is defined, we mean synonymously that there
exists a constant $C\geq 0$ such that $|f(x)|\leq Cg(x)$ for all $x\in
X$. The ``implied constant'' refers to any value of $C$ for which this
holds. It may depend on the set $X$, which is usually specified
explicitly, or clearly determined by the context.

\section{Methods from arithmetic geometry, I}
\label{sec:geoI}

In this section and the next we consider a field $k$ which is finitely
generated over the prime field, e.g., $k$ could be a number field or
a function field over a finite field.\footnote{\ These will be the
  only fields arising in the analytic section, and the reader can
  think of these as the most important.} We also assume that the
characteristic of $k$, if positive, is not equal to $2$. 
\par
The first conditions arise because we need to know that the following
mild weakening of Mordell's conjecture holds for $k$:

\begin{thm}\label{thm::mordell}
With $k$ as above, there is a constant $g_1(k)$ such that for any smooth
projective curve $C/k$ of genus $g> g_1(k)$, the set $C(k)$ of $k$-rational
points on $C$ is finite. 
\end{thm}

\begin{proof}
  At a minimum we must have $g\geq 2$, and if $\charac(k)=0$, then we
  may take $g_1(k)=2$.  If $C$ is not defined over an algebraic
  closure of the prime field of $k$, then this is a combination of
  results of Manin--Grauert \cite{manin}, \cite{grauert} (for
  $\charac(k)=0$) and Samuel \cite{samuel} (for $\charac(k)>0$).  If
  $\charac(k)=0$ and $C$ is defined over the algebraic closure of
  $\Q$, then the argument in the corollary of Theorem 1 of
  \cite{deschamps} reduces this to the celebrated theorem of Faltings
  \cite{faltings}.  The case which can force us to take $g_1(k)>2$ is
  when $k=\F_q(X)$ for a smooth projective variety $X/\F_q$ and $C$ is
  defined over $\F_q$.  If $\F_q$ is algebraically closed in $k$, then
  elements of the complement $C(k)-C(\F_q)$ correspond to dominant
  maps $X\to C$ and repeated composition with the Frobenius $C\to C$
  gives rise to an infinite subset of $C(k)$.  However, the following
  proposition shows if we take $g_1(k)=\dim H^0(X\times_{\F_q}\Fqbar,\Omega^1)$, there
  are no such elements, hence $C(k)=C(\F_q)$ is finite.
\end{proof}

\begin{prop}
  Let $Y/\Fqbar$ be a smooth projective curve of genus $g$. For any
  dominant map $f:X\to Y$ where $X/\Fqbar$ is a smooth projective variety, we
  have $g\leq \dim H^0(X,\Omega^1)$.
\end{prop}

The following proof was suggested by J.F.~Voloch.

\begin{proof}
  If $f:X\to Y$ is inseparable, then there is a purely inseparable map
  of curves $Z\to Y$ through which $f$ factors and such that $X\to Z$
  is separable.  Moreover, the genus of $Y$ is at most the genus of
  $Z$, so up to replacing $Y$ with $Z$ we may assume $f$ is separable.
  Then the pullback map of differentials 
$$
f^*:H^0(Y,\Omega^1)\to   H^0(X,\Omega^1)
$$
is an embedding (cf.~\cite[Theorem 1 in III.6.2]{shafarevich}), and since
$\dim(H^0(Y,\Omega^1))=g$, the conclusion follows.
\end{proof}

\par
Let now $C/k$ be a smooth curve, and let $A/k(C)$ be a principally-polarized
abelian variety of dimension $g$ over the function field of
$C$.  Let $\ell$ be a prime which is invertible in $k$ and let
$A[\ell]$ be the $\ell$-torsion of $A$. 
 
There is an embedding of the group $G=\Gal(k(C)(A[\ell])/k(C))$ into
$\Gamma=\Aut(A[\ell])$, where $\Aut$ is understood to refer to the
group of linear automorphisms preserving the symplectic Weil pairing,
up to a scalar.  The subgroup of symplectic automorphisms of $A[\ell]$
is denoted $\Gamma_0$. We therefore have isomorphisms
$$
\Gamma\simeq GSp(2g,\F_{\ell}),\quad\quad
\Gamma_0\simeq Sp(2g,\F_{\ell})
$$
(where $GSp(2g)$ is the group of symplectic similitudes, also
sometimes written $CSp(2g)$ or even $SSp(2g)$.)

By the {\em geometric monodromy} of $A$ {\em modulo $\ell$}, we mean
the image of the absolute Galois group of $k^s(C)$ in $\Gamma_0$.  We
say $A$ has {\em big monodromy mod $\ell$} if the geometric monodromy
of $A$ is the whole symplectic group $\Gamma_0$, so that $\Gamma_0
\leq G$.  If $v$ is a place of $k(C)$, we write $A_v$ for the fiber
over $v$ of the Neron model of $A$ over $C$ and $G_v\leq G$ for the
decomposition group.  We say $A_v$ has {\em big monodromy modulo
  $\ell$} if $A_v$ is an abelian $g$-fold and if $\Gamma_0 \leq
G_v\leq G$. In all this, if $\ell$ is clear from the context, we may
simply speak of {\em geometric monodromy}, or say that $A$ or $A_v$
 \emph{has big monodromy}, without specifying $\ell$.

 These notions are relevant for our basic problem because of the
 following sufficient criterion for geometric simplicity, which will
 be our main tool in this and the next section. This makes precise the
 fairly intuitive fact that a factorization of an abelian variety
 forces the monodromy group to preserve the factors, and hence is
 incompatible with having big monodromy; but because the factorization
 may exist only over an extension of $k$, and is valid only up to
 isogeny, this requires some care.

\begin{prop}\label{le:gs}
  For any $g\geq 1$, there is a constant $\ell_1(g)\geq 1$ satisfying
  the following: if $\ell>\ell_1(g)$ and $A/k$ is an abelian variety
  of dimension $g$ over a field $k$ such that $A$ has big monodromy
  modulo $\ell$, then $A$ satisfies $\End_{\bar{k}}(A)=\Z$ and in
  particular is geometrically simple.
\end{prop}

\begin{proof}
  By a theorem of Chow, we have $\End_{\bar{k}}(A)=\End_{k^s}(A)$ for
  any abelian variety $A/k$ (see~\cite[Th 3.19]{conr:langneron}), so
  it suffices to prove the corresponding statement with the
  endomorphism ring over $k^s$ instead of over $\bar{k}$.
  
  Next, for any $A/k$, note that the rank of the endomorphism ring
  $\End_{k^s}(A')$, as a $\Z$-module, is constant as $A'$ runs over
  the isogeny class of $A$. If $A$ is not geometrically simple, there
  is an abelian variety $A'$ in this isogeny class which splits over
  $\bar{k}$ as $A_1\times A_2$, with $A_1$, $A_2$ of dimension $\geq
  1$. By the previous paragraph, this means in particular that
  $\End_{k^s}(A')$ contains a non-trivial endomorphism $\pi$
  satisfying $\pi^2=\pi$ (e.g., the projection onto the non-trivial
  factor $A_1$), and then $\Z[\pi]$ is a rank-two $\Z$-submodule of
  $\End_{k^s}(A')$ and thus $\End_{k^s}(A)\neq\Z$ (since it has rank
  $\geq 2$).  In particular, by contraposition, $A$ is geometrically
  simple if $\End_{k^s}(A)=\Z$.
  
  Now, let $\ell$ be a prime number such that some abelian variety
  $A/k$ has big monodromy modulo $\ell$ and satisfies
  $\End_{k^s}(A)\neq\Z$.  Then, by the theory of abelian groups, there
  is an endomorphism $\psi$ in $\End_{k^s}(A)$ such that $\Z[\psi]$ is
  a rank-two $\Z$-submodule of $\End_{k^s}(A)$ and moreover
  $\End_{k^s}(A)/\Z[\psi]$ has no $\ell$-torsion.  The latter
  assumption implies that the image of $\Z[\psi]$ in
  $\End(A[\ell])\simeq M_{2g}(\F_\ell)$ is a rank-two
  $\F_\ell$-submodule, because otherwise $\psi-m$ would be divisible
  by $\ell$ for some $m\in\Z$.  More precisely, we may find $\psi$
  such that the image of $\psi$ in $\End(A[\ell])$ does not lie in the
  scalar subgroup $\F_\ell^\times$.

  Let $K$ be the Galois closure of the splitting field of $\psi$
  (i.e., $K$ is the fixed field of the subgroup of $\Gal(\bar{k}/k)$
  fixing $\psi$) and let $H$ be its Galois group of $K(A[\ell])/K$.
  There is a natural inclusion $H\to G$, where $G$ is the monodromy
  group of $A$ modulo $\ell$.

  Since the action of $\psi$ on $A[\ell]$ commutes with $H$ and $\psi$
  does not lie in the scalar subgroup
  $\F_\ell^\times\leq\End(A[\ell])$, Schur's Lemma implies that the
  subgroup $H\leq M_{2g}(\F_\ell)$ does not act absolutely irreducibly
  on $A[\ell]$. Since $G\cap\Gamma_0=\Gamma_0$ does have this property
  (because of the big monodromy assumption), $H\cap\Gamma_0$ is a proper
  subgroup of $\Gamma_0$.  Now, if
  $\ell>3$, we know that $\Gamma_0$ is generated by its elements of order
  $\ell$, because they generate a normal subgroup and $Z(\Gamma_0)=\{\pm 1\}$
  is the only proper normal subgroup (see \cite[Theorem 5]{stei}).
  Thus, there exists at least one element $\sigma$ of order $\ell$ in
  the complement $G-H$.  In particular, the $\sigma$-orbit of $H$ in
  the permutation representation on $G/H$ has $\ell$ elements, hence
  we find that $[G:H]\geq\ell$.

  On the other hand, the Galois group $\Gal(K/k)$ acts faithfully on
  the free $\Z$-module $\End_K(A)$, so that it is isomorphic to a
  finite subgroup $F$ of $\GL(n,\Z)$ for some $n\leq 2g$.  By a
  theorem of Minkowski, $F$ injects into $\GL(n,\Z/3\Z)$ (see for
  instance \cite{silv:sz}) and thus its order is bounded by a constant
  depending only on $g$. Let $\ell_1(g)$ be this constant. Since
  Galois theory gives
$$
[G:H]\leq |\Gal(K/k)|,
$$
it follows from this and the previous paragraph that
$$
\ell\leq [G:H]\leq |F|\leq \ell_1(g),
$$
as desired.
\end{proof}

Our first (and most general) approach to the problem mentioned in the
introduction uses some deep group-theoretic results of Liebeck--Saxl
\cite{ls} and Guralnick \cite{gura:monodromycharp}, in order to apply
Proposition~\ref{le:gs}. This is contained in the following result:

\begin{prop}\label{le:lsg}
  If $g_1\geq 0$ is a constant, then there is a constant $\ell_2(g_1)$
  satisfying the following.  If $\ell>\ell_2(g_1)$ and $X\to C$ is a
  geometric Galois cover with group $G=\Sp(2g,\F_{\ell})$, then for any
  proper subgroup $H<G$, the genus of $X/H$ is at least $g_1$.
\end{prop}

\begin{proof}
  In the case where $f$ is tamely ramified (for instance in
  characteristic zero), this follows from \cite[Corollary 2 to Theorem
  1]{ls}, and in the general case, this follows from \cite[Theorem
  1.5]{gura:monodromycharp}.
\end{proof}

\begin{rem}
  The constant $\ell_2(g_1)$ is
  conjectured to be independent of $g_1$ (\cite[Conjecture
  1.6]{gura:monodromycharp}), and in the tame case this follows from
  \cite[Theorem A]{froh:frohmagannals}.
\end{rem}

What is required for Proposition~\ref{le:lsg} is a very thorough
understanding of the maximal proper subgroups of $\Sp(2g,\F_{\ell})$.
As written, the results in \cite{ls} and \cite{gura:monodromycharp}
both use the classification of finite simple groups.  More precisely,
the proof of Corollary 9.5 in \cite{gura:monodromycharp} uses Theorem
1 of \cite{ls} which in turn rests on the classification-dependent
Theorem 4.1 of \cite{lie}.  However, we learned from Guralnick
\cite{gura:personal} that Magaard has an unpublished proof of Theorem
1 of \cite{ls} which does not use the classification.

Proposition~\ref{le:lsg} forms the main content of the following
proposition.

\begin{prop}
If $\ell>\ell_2(g_1(k))$ and $A$ has big monodromy mod $\ell$, then $A_v$
has big monodromy mod $\ell$ for all but finitely many $v\in C(k)$.
\label{pr:general}
\end{prop}

\begin{proof}
Let $X/k$ be the smooth curve with function field $k(C)(A[\ell])$.  The map
of curves $X\to C$ is generically Galois with group $G$ containing
$\Gamma_0$.  Let $v$ be a point in $C(k)$ and let $w$ be a point in $X$
lying over $v$ with decomposition group $G_v\leq \Gamma$.  If $H\leq G$ is
a subgroup not containing $\Gamma_0$,  and $G_v\leq H$, then the image of
$w$ in the quotient curve $X/H$ has degree $[G_v:G_v\cap H]=1$ over $v$,
hence is a $k$-rational point of $X/H$.  In particular, to prove the
theorem it suffices to show that $X/H$ has genus greater than $g_1(k)$ for any
proper subgroup $H<G$ because then Theorem~\ref{thm::mordell} implies
that
$$
\bigcup_{H<G}\,(X/H)(k)
$$ 
is finite.  But this is exactly Proposition~\ref{le:lsg} applied to
the proper subgroup $H \cap \Gamma_0$ of $\Gamma_0$.
\end{proof}

We can now deduce the following concrete application:

\begin{thm}\label{th:main}
  Let $k$ be an infinite field of finite type over the prime field,
  for instance a number field. Let $g\geq 1$ be an integer, and let
  $f\in k[X]$ be a squarefree polynomial of degree $2g$.
\par
Let $A$ be the Jacobian of the hyperelliptic curve of genus $g$ over
$k(t)$ with affine model
$$
y^2 = f(x)(x-t).
$$
Then there are only finitely many $t\in k$ such that $A_t$ is not
geometrically simple.
\end{thm}

\begin{proof}
  By a result of J-K. Yu and the third author~\cite{hall:monodromy},
  $A$ has big monodromy modulo $\ell$ for any $\ell\geq 3$. Choosing
  $\ell > \max(2,\ell_1(g),\ell_2(g_1(k)))$ yields the desired result by combining
  Proposition~\ref{le:gs} and Proposition~\ref{pr:general}.
\end{proof}

In the theorems above we have used the fact that $A$ has big monodromy
modulo some prime $\ell$ in order to show that almost all the fibers
$A_v$ have big monodromy modulo the same $\ell$.  It is worth pointing
out that the hypothesis that $A_v$ has big monodromy modulo a
sufficiently large fixed $\ell_0$ actually implies that it has big
monodromy modulo almost all $\ell$, although we will only prove it for
global fields.

\begin{prop}
  Suppose $k$ is a global field, i.e.~a number field or a function
  field of a curve over a finite field.  If $A_v$ has big monodromy
  modulo $\ell_0$, for some $\ell_0\geq 5$, then there is a constant
  $\ell_3(A_v)$ so $A_v$ has big monodromy modulo $\ell$ for every prime
  $\ell>\ell_3(A_v)$.
\end{prop}

\begin{proof}
  If $A_v$ has big monodromy for $\ell_0\geq 5$, then the $\ell_0$-adic
  monodromy group of $A_v$ contains $\Aut(T_{\ell_0}
  A)\simeq\Sp(2g,\Z_{\ell_0})$ (see \cite[Lemme 1]{serre:vigneras}).
  Therefore, if $k$ is a number field, then \cite[2.2.7]{serre:resume}
  and \cite[Th\'eor\`eme 3]{serre:vigneras} imply that for every
  sufficiently large $\ell$, the $\ell$-adic monodromy group of $B$
  contains $\Sp(2g,\Z_\ell)$.  If $k$ is a function field over a
  finite field, then one can apply \cite[8.2]{serre:vigneras} to
  deduce a similar statement.
\end{proof}

It is worth noting here that this method does {\em not} allow
the bound $\ell_3(A_v)$ to be chosen independently of $A_v$.  To prove
such a uniform bound over a rational function field, for example, would require showing that the Siegel modular
varieties parametrizing abelian $g$-folds with `$H$-level structure'
contain no unexpected rational curves; this can be carried out when $g=1$,
since the Siegel modular variety is just a curve (see~\cite{cojo:cojohall})
but seems difficult in general.  A theorem of
Nadel~\cite{nadel:nonexistence} proves such a result (as a special case of
a much more general theorem) when $H$ is the trivial subgroup of
$\Sp(2g,\F_\ell)$.

\section{Methods from arithmetic geometry, II}
\label{sec:geoII}

In the special case of families of hyperelliptic curves contemplated
in the present paper, we can also obtain results using easier group
theory in place of Proposition~\ref{le:lsg}, as we now explain.  Again,
we will use Proposition~\ref{le:gs} to obtain geometric simplicity.

We continue with the notation introduced in the previous section
except that now we must work in characteristic zero, so we assume $k$
is a finitely generated over a number field. This implies that
Theorem~\ref{thm::mordell} is valid with $g_1(k)=2$.
\par
First of all, we remark that when $A$ has big monodromy modulo a
sufficiently large $\ell$ and at least three fibers where the reduction
is not potentially good, then one can
show that $A_v$ has big monodromy modulo $\ell$ via the results in
\cite{hall:maximal}, which require only Thompson's classification of
so-called quadratic pairs \cite{thom}.

By restricting $A$ further, we can make our work even simpler, while
still proving a general enough result to obtain the theorems stated in
the introduction.  For this, we say $A$ {\em degenerates simply} at
$v$ if the identity component of $A_v$ is the extension of an abelian
variety by a one-dimensional torus and if the component group of $A_v$
has order prime to $\ell$.  There are only finitely many $v$ where $A$
degenerates simply.  From the group-theoretic point of view, this
geometric condition is useful because of the following fact:

\begin{lem}\label{lm:degener}
  With notation as above, if $A$ degenerates simply at $v$, then the
  inertia group $I_v\leq G_v$ is generated by a transvection.
\end{lem}

\begin{proof}
  By \cite[(2.5.4) and Corollaire 3.5.2]{sga7}, $I_v$ is generated by
  a unipotent element $\tau$ satisfying $\dim((\tau-1)A[\ell])\leq 1$,
  so $\tau$ is either a transvection or is trivial.  Moreover,
  $A[\ell]$ does not split over the strict henselization of the local
  field $k(C)_v$ because the component group of $A_v$ has order prime
  to $\ell$ (cf.~\cite[(11.1.3)]{sga7}), hence $k(C)(A[\ell])$
  ramifies over $v$ and $\tau\neq 1$ is a transvection, as claimed.
\end{proof}

We will also use here the following group-theoretic lemma, the
potential significance of which is clear from the previous one.

\begin{lem}\label{le:tr}
  If $\ell\geq 3$, then a subgroup of $Sp(2g,\F_{\ell})$ which
  contains $\ell^{2g-1}$ transvections is the whole of
  $Sp(2g,\F_{\ell})$.
\end{lem}
 
\begin{proof} This follows immediately from a theorem of Brown and
  Humphries~\cite{brow:sympgen}, which gives a criterion for a set of
  transvections to generate the symplectic group $Sp(2g,\F_{\ell})$.
  More precisely, recall that there is a natural bijection between
  cyclic groups generated by transvections and lines in
  $\F_\ell^{2g}$; namely, we take the group generated by $\tau$ to the
  $1$-dimensional space $(\tau-1)(\F_\ell^{2g})$.  Let $S \subset
  \P(\F_\ell^{2g})$ be a set of subgroups generated by transvections.
  Let $G(S)$ be the graph with set of vertices $S$ and with edges
  given by those pairs $(s_1,s_2)\in S\times S$ such that the space
  spanned by $s_1$ and $s_2$ (thought of as lines in $\F_\ell^{2g}$)
  is not isotropic.  Then~\cite{brow:sympgen} shows that (for
  $\ell\geq 3$), $S$ generates $G$ if and only if the elements of $S$
  span $\F_\ell^{2g}$, and if $G(S)$ is connected.  If the lines in
  $S$ fail to span all of $\F_\ell^{2g}$, then obviously
$$
|S| \leq \frac{\ell^{2g-1}-1}{\ell-1}.
$$
\par
On the other hand, if $G(S)$ is the disjoint union of two subgraphs
$G_1$ and $G_2$, the subspaces of $\F_\ell^{2g}$ spanned by the
vertices of $G_1$ and $G_2$ must be mutually orthogonal, so in
particular the union of these vector spaces contains at most
$(\ell^{2g-1}-1)/(\ell-1)$ lines.  In either case, the number of
transvections contained in $S$ is at most $\ell^{2g-1}-1$.
\end{proof}

Now we deduce the following:

\begin{prop}
\label{pr:main}
Let $k$ be a field finitely generated over a number field, let $C/k$
be a smooth projective curve, and let $A/k(C)$ be a
principally-polarized abelian $g$-fold.  Suppose $\ell\geq 3$ is a
prime such that $A$ has big monodromy modulo $\ell$ and that $A$
degenerates simply at
$$
\Bigl\lceil \frac{2(\ell^{2g}-1)}{(\ell^g-\ell^{g-1})^2}\Bigr\rceil
$$ 
or more places.  Then $A_v$ has big monodromy modulo $\ell$ for all
but finitely many $v\in C(k)$.
\end{prop}

\begin{proof} 
  We can assume that the places where $A$ degenerates simply are in
  $C(k)$, because the conclusion will even be stronger after extending
  scalars to a field of definition of those places.  Then, let again
  $X/k$ be the smooth curve with function field $k(C)(A[\ell])$.  The
  map of curves $X\to C$ is generically Galois with group $G$
  contained in $\Gamma$.  Again, we use Theorem~\ref{thm::mordell},
  applied to the curves $X/H$ as $H$ ranges over proper subgroups of
  $\Gamma_0$.  As in the proof of Proposition~\ref{pr:general}, and
  because $g_1(k)=2$ now, it suffices to show that all such $X/H$ have
  genus at least $2$.
\par
Fix a proper subgroup $H<\Gamma_0$ and let $Y/k$ be the quotient curve
$X/H$.  Suppse $v$ is a point where $A$ degenerates simply and let
$\tau\in I_v$ be a generator.  There is an action of $\tau$ on the
sheets of $Y\times_k k^s$ which is exactly the permutation action on
the cosets of $\Gamma_0/H$: the orbits correspond to the points of
$Y\times_k k^s$ over $v$ and the size of an orbit is the ramification
index.  Every orbit has 1 or $\ell$ elements and the coset $gH$ is
fixed by $\tau$ if and only if $g^{-1}\tau g$ lies in $H$.  In
particular, the computation of the ramification of $Y \to C$ at $v$ is
reduced to a problem about the conjugates of transvections in $G$.

By Lemma~\ref{le:tr}, we have
$$
\frac{|\tau^G\cap H|}{|\tau^G|}\leq
\frac{\ell^{2g-1}-1}{\ell^{2g}-1},
$$
so there are at least 
$$
\frac{\ell^{2g-2}(\ell-1)}{\ell^{2g}-1}[G:H]
$$
points of $Y\times_k k^s$ over $v$ of ramification degree $\ell$.
Therefore, if we write $m$ for the number of $v$ in $C(k)$ where $A$
degenerates simply, then from the Riemann-Hurwitz formula we have that
$$
2g(Y) - 2 \geq [G:H]\Bigl(
\frac{m\ell^{2g-2}(\ell-1)^2}{\ell^{2g}-1} +
2g(C) - 2\Bigr).
$$
\par
In particular, the right hand side is positive since
$m(\ell^g-\ell^{g-1})^2> 2(\ell^{2g}-1)$, hence $Y=X/H$ has genus at
least two.
\end{proof}

\begin{exmp}\label{ex:main}
When $A$ is the Jacobian of 
$$
y^2 = f(x)(x-t)
$$
with $\deg(f) = 2g$, we observe that, for $\ell\geq 3$, $A$
degenerates simply at every prime $v$ in $k(t)$ corresponding to the
specialization of $t$ to a root of $f(x)$.  A priori, one could apply
the description of the monodromy of $A$ about $v$ given in
\cite[Section 5]{hall:maximal} to deduce that it is a transvection,
which is why we want it to be simply degenerate (see
Lemma~\ref{lm:degener}), but one can also perform a geometric
computation to check this directly.
\par  
The fact that $A_v$ is the extension of an abelian variety by a
one-dimensional torus, for instance, from~\cite[\S 9.2, Example
8]{blr}.  The key point is that the fiber of the curve over $v$ is
smooth away from a single ordinary double point.
\par 
To compute the order of the component group of $A_v$, one must compute
the minimal regular model of the curve over $v$, which a
straightforward calculation reveals to be the union of curve $C_1$ of
genus $g-1$ and a curve $C_2$ of genus $0$ (Remark IV.7.7 and Example
IV.7.7.1 of \cite{silverman2} give a nice concrete treatment of the
blowing-up process required for this computation).  Moreover, $C_1$
and $C_2$ intersect in two points, from which it follows that one has
the divisor intersection numbers $C_1^2=C_2^2=-2$ and $C_1\cdot C_2=2$
(cf.~\cite[Proposition IV.8.1]{silverman2}).  Using this information
one applies \cite[\S 9.6, Theorem 1]{blr} to deduce that the component
group of $A_v$ is isomorphic to $\Z/2\Z$.
\par 
So when $g \geq 2$, we immediately recover Theorem~\ref{th:main} using
Proposition~\ref{pr:main}.  (The case $g=1$ is standard; see for
instance \cite{cojo:cojohall}.)
\end{exmp}

\section{Methods from analytic number theory}
\label{sec:analytic}

The analytic approach to our problem is based on the conjunction of
two sieves: the sieve for Frobenius of the last-named author
(see~\cite{kow:ls}), which is a version of the large sieve, and a
generalisation of Gallagher's larger sieve~\cite{gallagher}.  The
prototype of this approach was described in~\cite[Prop.  6.3]{kow:ls},
which used a standard large sieve instead of the larger sieve. The
latter is much more efficient here.
\par
This combination of two sieves is quite appealing, and it may be of
interest in other applications. Although we do not know of any previous
use of the large sieve to set up a larger sieve, the second-named
author has, in earlier work, used the larger sieve to prepare for
application of the large sieve (see~\cite{elsholtz}).
\par
The sieve arises because, instead of the ``big monodromy'' argument
in Proposition~\ref{le:gs}, we will detect non-simple abelian varieties
by means of the following alternate criterion:

\begin{prop}\label{pr:alternate}
  Let $k$ be a number field and $A/k$ be an abelian variety. Let
  $\ideal{p}\subset \Z_k$ be a prime ideal of $k$ with residue field
  $\F_{\ideal{p}}$ such that $A$ has good reduction at $\ideal{p}$. If
  the abelian variety $A_{\ideal{p}}/\F_{\ideal{p}}$ obtained by
  reduction of $A$ modulo $\ideal{p}$ is geometrically simple, then so
  is $A$.
\end{prop}

\begin{proof}
  This is a tautology, given the theory of reductions of abelian
  varieties: if $A$ is not geometrically simple, there exists an
  isogeny
$$
A\simeq A_1\times A_2
$$
with $\dim A_1$, $\dim A_2\geq 1$, which is defined over some finite
Galois extension $k'/k$. The factors $A_1$ and $A_2$ have good
reduction at $\ideal{p}$, and so, after reducing, we obtain a
corresponding non-trivial factorization for $A_{\ideal{p}}$ defined
over the residue extension of $k'/k$ at $\ideal{p}$.
\end{proof}

\begin{rem}
  It is well-known that there exist integral polynomials which are
  irreducible over $\Q$ but which are reducible modulo every prime
  (this is due to Hilbert; see, e.g.,~\cite{brandl}, where it is shown
  that such polynomials exist of every non-prime degree).  Similarly,
  there are examples of geometrically simple abelian varieties defined
  over a number field which are not geometrically simple modulo any
  prime (see the review~\cite{fisher} by R. Fisher of a paper by C.
  Adimoolam, and the results of Hashimoto and
  Murabayashi~\cite{hm}). It would be interesting to know if the
  analogue of the finiteness statement~(\ref{eq:bound-geometric})
  holds for the set $S'(B)$ of parameters of height $\leq B$ for which
  $A_t$ is not simple modulo all primes.
\end{rem}

Sieve methods, in particular the large sieve, will be used to detect
factorizations of abelian varieties over finite fields (much as they
can be used to detect irreducible polynomials), and thus we will
proceed by applying Proposition~\ref{pr:alternate} at many different
primes.

We first give a new formulation of Gallagher's sieve in number fields
(the works of Hinz~\cite{hinz} and Goldberg~\cite{goldberg} have other
versions, as does a work in progress of D. Zywina). The terminology
``larger sieve'' arises because this statement is most efficient when
trying to control the size of a set which does not intersect a very
large number of residue classes modulo a set of primes.

\begin{prop}\label{pr:larger}
  Let $k/\Q$ be a number field, let $B>0$ be a constant, and let
  $\mathcal{A}$ be a finite set of elements of $k$ such that $H(a)\leq
  B$ for all $a\in\mathcal{A}$, where $H$ denotes the height in $k$,
  normalized as described below.
  \par
  Let $S$ be a finite set of prime ideals in the ring of integers
  $\Z_k$. If the order of the image of $\mathcal{A}$ under the
  reduction map $k\rightarrow \P^1(\F_{\ideal{p}})$ is $\leq
  \nu(\ideal{p})$ for all $\ideal{p}\in S$, then we have
$$
|\mathcal{A}|
\leq
\frac{\displaystyle{
\sum_{\ideal{p}\in S}\log N\ideal{p}-\log (2^{[k:\Q]}B^2)}
}
{\displaystyle{\sum_{\ideal{p}\in S}
{\frac{\log N\ideal{p}}{\nu(\ideal{p})}}-
\log B -\log(2^{[k:\Q]}B^2)}},
$$
\emph{provided} the denominator in either of these two expressions is
positive.
\end{prop}

\begin{rem}
For many applications, the weaker estimate
\begin{equation}\label{eq-larger}
|\mathcal{A}|\leq
\frac{\displaystyle{
\sum_{\ideal{p}\in S}\log N\ideal{p}}
}
{\displaystyle{\sum_{\ideal{p}\in S}
{\frac{\log N\ideal{p}}{\nu(\ideal{p})}}-2\log(2^{[k:\Q]}B^2)}},
\end{equation}
also valid when the denominator is positive, is sufficient. Indeed,
this is what we will use. 
\end{rem}

We  indicate which definition of the height we consider, since
there are competing normalizations; we
follow~\cite[VIII.5]{silverman}, i.e., our $H$ is the same as
Silverman's $H_k$. Thus let $M_k$ be the set of places of $k$, defined
as in~\cite[VIII.5, p. 206]{silverman} (the set of absolute values on
$k^{\times}$, which coincide with the standard absolute values on $\Q$
when restricted to $\Q^{\times}$), and let $|\cdot|_v$ denote the
absolute value associated with $v\in M_k$.
\par
For $a\in k$, the height of $a$ is defined by
$$
H(a)=\prod_{v\in M_k}{\max(1,|a|_v^{n_v})}
$$
where $n_v$ is the local degree at $v$, i.e., $n_v=[k_v:\Q_v]$, where
$k_v$ and $\Q_v$ are the completions of $k$ (resp. $\Q$) with respect
to the metric defined by $\|\cdot\|_v$ (in particular $n_v=2$ if $v$
is a complex place).
\par
We will need the following easy and well-known results:
\begin{equation}
\label{eq:height}
H(a)=H(a^{-1})\quad\quad
H(ab)\leq H(a)H(b)\quad\quad
H(a+b)\leq 2^{[k:\Q]}H(a)H(b)
\end{equation}
for all $a$, $b\in k^{\times}$. We also recall that if $v\in M_k$ is a
non-archimedean place, associated with a prime ideal $\ideal{p}$, then
we have
\begin{equation}\label{eq-padic-height}
|a|_{v}^{n_v}=(N\ideal{p})^{-v_{\ideal{p}}(a)},
\end{equation}
where $v_{\ideal{p}}$ is the $\ideal{p}$-adic valuation and
$N\ideal{p}=|\F_{\ideal{p}}|=|\Z_k/\ideal{p}\Z_k|$ is the order of
the residue field. 
\par
We also comment briefly on the reduction map $k\ra
\P^1(\F_{\ideal{p}})$: if $a\in k$ and $v_{\ideal{p}}(a)<0$ (i.e., if
$\ideal{p}$ ``divides the denominator'' of $a$), then the image of $a$
modulo $\ideal{p}$ is the point at infinity (denoted $\infty$) in
$\P^1(\F_{\ideal{p}})$. We write simply $a\equiv
\infty\mods{\ideal{p}}$ to indicate that this is the case.

\begin{proof}[Proof of Proposition~\ref{pr:larger}]
  The proof is very similar to the original argument of
  Gallagher~\cite{gallagher}. Let
$$
\Delta=\prod_{\stacksum{a,b\in \mathcal{A}}{a\not=b}}{H(a-b)}
$$
which is real number $\geq 1$. We will compare upper and lower bounds
for $\Delta$ to obtain the larger sieve inequality.
By~(\ref{eq:height}), we first have the easy upper bound
\begin{equation}\label{eq-1}
\Delta\leq (2^{[k:\Q]}B^2)^{|\mathcal{A}|(|\mathcal{A}|-1)}.
\end{equation}
\par
On the other hand, we bound the height from below as follows:
by~(\ref{eq:height}) again, switching to the inverse to
use~(\ref{eq-padic-height}) with positive valuations, we have
$$
\Delta=
\prod_{a\not=b}{H((a-b)^{-1})}\geq 
\prod_{a\not=b}{
\prod_{\ideal{p}\in S_{a,b}}{
(N\ideal{p})^{v_{\ideal{p}}(a-b)}}}
$$
where
$$
S_{a,b}=\{\ideal{p}\in S\,\mid\, v_{\ideal{p}}(a)\geq 0,
\quad v_{\ideal{p}}(b)\geq 0\}.
$$
\par
It follows that
\begin{align*}
\log\Delta&
\geq \sum_{a\not=b}{
\sum_{\stacksum{\ideal{p}\in S_{a,b}}
{a\equiv b\mods{\ideal{p}}}}
{
(\log N\ideal{p})
}
}\\
&=
\sum_{a\not=b}{
\sum_{\stacksum{\ideal{p}\in S}
{a\equiv b\mods{\ideal{p}}}}
{
(\log N\ideal{p})
}
}
-
\sum_{a\not=b}{
\sum_{\stacksum{\ideal{p}\in S}{a\equiv b\equiv \infty\mods{\ideal{p}}}}
{
(\log N\ideal{p})
}
}\\
&= L_1-L_2,\quad\quad\text{(say)},
\end{align*}
(since if $\ideal{p}\in S-S_{a,b}$, then $a\equiv b\mods{\ideal{p}}$
implies that both $a$ and $b$ reduce to $\infty$).
\par
Now, for all $\ideal{p}\in S$ and $\alpha\in \P^1(\F_{\ideal{p}})$,
define
$$
R_{\ideal{p}}(\alpha)=
|\{a\in\mathcal{A}\,\mid\, a\equiv \alpha\mods{\ideal{p}}\}|.
$$
\par
We obtain
\begin{align*}
L_1&=\sum_{\ideal{p}\in S}{(\log N\ideal{p})}
{
\sum_{\stacksum{a\not=b}{a\equiv b\mods{\ideal{p}}}}{1}
}\\
&=
\sum_{\ideal{p}\in S}{(\log N\ideal{p})
\sum_{\stacksum{a,b\in\mathcal{A}}{a\equiv b\mods{\ideal{p}}}}{1}}
-|\mathcal{A}|\sum_{\ideal{p}\in S}{\log N\ideal{p}}\\
&=
\sum_{\ideal{p}\in S}{(\log N\ideal{p})
\sum_{\alpha\in\P^1(\F_{\ideal{p}})}{R_{\ideal{p}}(\alpha)^2}}
-|\mathcal{A}|\sum_{\ideal{p}\in S}{\log N\ideal{p}}.
\end{align*}
\par
However, by Cauchy-Schwarz, and by definition of $\nu(\ideal{p})$, we
have the familiar lower bound
$$
\sum_{\alpha\in \P^1(\F_{\ideal{p}})}{R_{\ideal{p}}(\alpha)^2}\geq
\frac{\Bigl(\displaystyle{\sum_{\alpha\in
      \P^1(\F_{\ideal{p}})}{R_{\ideal{p}}(\alpha)}}\Bigr)^2}
{\nu(\ideal{p})} =\frac{|\mathcal{A}|^2}{\nu(\ideal{p})},
$$
and therefore we obtain
$$
L_1\geq \sum_{\ideal{p}\in S}{
\Bigl\{
\frac{|\mathcal{A}|^2}{\nu(\ideal{p})}-|\mathcal{A}|
\Bigr\}
\log N\ideal{p}}.
$$
\par
Now we bound $L_2$ from above in order to conclude; we have for all
$a$ and $b$ in $\mathcal{A}$ the rather trivial estimate
$$
\sum_{\stacksum{\ideal{p}\in S}{a\equiv b\equiv \infty\mods{\ideal{p}}}}
{
(\log N\ideal{p})
}\leq 
\sum_{\stacksum{\ideal{p}\in S}{a\equiv \infty\mods{\ideal{p}}}}
{
(\log N\ideal{p})
}\leq
\log H(a)\leq \log B,
$$
and finally by putting things together, we obtain
$$
\sum_{\ideal{p}\in S}\Bigl\{
\frac{|\mathcal{A}|^2}{\nu(\ideal{p})}-|\mathcal{A}|
\Bigr\}
\log N\ideal{p}
-|\mathcal{A}|^2(\log B)
\leq \log H(\Delta)\leq
|\mathcal{A}|(|\mathcal{A}|-1)\log (2^{[k:\Q]}B^2).
$$
\par
Simplifying by $|\mathcal{A}|$ and re-arranging gives the result.
\end{proof}

When applying this proposition, we assume some upper bound on
$\nu(\ideal{p})$, on average over $S$, and estimate the right-hand
side of~(\ref{eq-larger}). In our case, $\nu(\ideal{p})$ will be quite
small (less than $(N\ideal{p})^{1-\delta}$ for some $\delta>0$), so
that if the set $S$ is chosen to be
$$
S=\{\ideal{p}\subset \Z_k\,\mid\, N\ideal{p}\leq x\}
$$
for some parameter $x\geq 2$ (as is typically the case), the first sum
in the denominator grows fairly rapidly as $x$ grows.
\par
The strength of the final estimates stems from this, but in a way
which is rather surprising compared with the large sieve (for
instance): it will come from the fact that one can choose $x$ quite
small to make the denominator positive; then the numerator is also
fairly small, hence so is $\mathcal{A}$, but the actual size of the
denominator is, in fact, of little significance (in other words, it is
not really a ``saving factor''). 
\par
From this sketch, one can guess that the only really delicate issue
that may arise is if one tries to have estimates uniform in terms of
$k$, for then one is led directly to the difficult issue of showing
that there are sufficiently many prime ideals with small norm. 
\par
In order to clarify the mechanism, we define 
\begin{equation}\label{eq-pit}
\pit{k}(x;\delta)=
\min\Bigl\{ t\geq 2\,\mid\,
\sum_{N\ideal{p}\leq t}{(N\ideal{p})^{-1+\delta}}\geq x
\Bigr\},
\quad\quad \text{ for } t\geq 2,\quad 0\leq \delta<1,
\end{equation}
which, intuitively, quantifies the ``convergence to equilibrium'' in
the Prime Ideal Theorem for $k$. Note in particular that
$$
\pit{k}(x;\delta)\geq \min\{n\geq 2\,\mid\, \text{
there is some prime ideal of norm $n$}\},
$$
since \emph{any} sum over primes of smaller norm is zero by
definition. 
\par
If $k$ is considered to be fixed, we can deduce, by summation by parts,
from the Prime Ideal Theorem that
$$
\sum_{N\ideal{p}\leq t}{(N\ideal{p})^{-1+\delta}}=
\frac{t^{\delta}}{
\log t^{\delta}}+O\Bigl(\log \frac{1}{\delta}+
\frac{t^{\delta}}{(\log t^{\delta})^2}\Bigr),
$$
for $\delta>0$ and $t\geq 2$ with $t^{\delta}\geq 2$, where the
implied constant depends on $k$ only. It then follows easily that
\begin{equation}\label{eq-fixed}
\pit{k}(x;\delta)\ll 
(2x\log x)^{1/\delta}
\end{equation}
for $x\geq 2$, where the implied constant depends only on $k$.

\begin{cor}\label{cor:1}
  Let $k/\Q$ be a number field and let $\mathcal{A}$ be a finite set
  of elements of $k$ such that $H(a)\leq B$ for all $a\in\mathcal{A}$,
  and such that, for all prime ideals $\ideal{p}$ in $\Z_k$, the order
  of the image of $\mathcal{A}$ under the reduction map $k\rightarrow
  \P^1(\F_{\ideal{p}})$ is $\leq \nu(\ideal{p})$ where
$$
\nu(\ideal{p})\leq C(N\ideal{p})^{1-\gamma^{-1}}(\log N\ideal{p})
$$
for some constants $C> 0$ and $\gamma\geq 1$.
\par
Then we have
$$
|\mathcal{A}|\leq 
2C[k:\Q]\pit{k}\Bigl(3C\log(2^{[k:\Q]}B^2);\gamma^{-1}\Bigr)
(\log 2^{[k:\Q]}B^2)^{-1}.
$$
\end{cor}

\begin{proof}
  Write $\delta=\gamma^{-1}$. Applying Proposition~\ref{pr:larger} (in
  the form of~(\ref{eq-larger})) with $S$ taken to be the set
$$
S=\{\ideal{p}\,\mid\, N\ideal{p}\leq x\}
$$
for some $x\geq 2$ to be determined later, the denominator
of~(\ref{eq-larger}) is
$$
-2\log(2^{[k:\Q]}B^2)
+\sum_{\ideal{p}\in S}
{\frac{\log N\ideal{p}}{\nu(\ideal{p})}}
\geq -2\log (2^{[k:\Q]}B^2)+
C^{-1}\sum_{N\ideal{p}\leq x}{(N\ideal{p})^{-1+\delta}}.
$$
\par
Thus if we take
$$
x=\pit{k}\Bigl(3C\log(2^{[k:\Q]}B^2);\delta\Bigr),
$$
then the definition~(\ref{eq-pit}) shows that the denominator is 
$\geq \log(2^{[k:\Q]}B^2)$.
\par
We bound the numerator, on the other hand, rather wastefully in terms
of $k$:
$$
\sum_{N\ideal{p}\leq x}\log N\ideal{p}
\leq [k:\Q](\log x)\pi(x)\leq 2[k:\Q]x,
$$
(by the Brun-Titchmarsh or Chebychev upper-bound for $\pi(x)$). The
result is then a direct translation of  Proposition~\ref{pr:larger}.
\end{proof}

Under various assumptions, one can easily transform this into concrete
results. For simplicity, we do this for a fixed number field; there, 
using~(\ref{eq-fixed}), we obtain:

\begin{cor}\label{cor:2}
  Let $k$ be a fixed number field. With assumption as in
  Corollary~\ref{cor:1}, we have
$$
|\mathcal{A}|\ll (\log 2^{[k:\Q]}B^2)^{\gamma-1}(6C\log (9C\log
2^{[k:\Q]}B))^{\gamma},
$$
for all $B\geq 2$, the implied constant depending only on $k$.
\end{cor}

\begin{exmp}
For $k=\Q$, using a lower-bound such as
$$
\pi(x)\geq \frac{1}{6}\frac{x}{\log x}
$$
for $x\geq 2$ (which follows, e.g., from~\cite[p. 342]{hardy-wright}),
one gets easily (and rather wastefully) that
$$
\beta_{\Q}(x;\delta)\leq \Bigl(\frac{12x}{\delta}\log
\frac{2x}{\delta}\Bigr)^{1/\delta},
$$
and hence
$$
|\mathcal{A}|\leq 2\Bigl(\frac{36C}{\delta}\Bigr)^{1/\delta}
(\log 2B^2)^{1/\delta-1}
\Bigl(\log \frac{6C}{\delta}\log 2B^2\Bigr)^{1/\delta},
$$
under the assumption of Corollary~\ref{cor:1} for $k=\Q$.
\end{exmp}




Now we come to the application to the splitting of Jacobians in our
hyperelliptic families. We use the following result, which is itself
proved using a version of the large sieve, to derive assumptions such
as those in Corollary~\ref{cor:1}, involving the type of conditions in
Proposition~\ref{pr:alternate}.

\begin{prop}\label{pr:frobsieve}
  Let $\F_q$ be a finite field with $q$ elements, let $g\geq 1$ be an
  integer and let $f\in \F_q[X]$ be a squarefree polynomial of degree
  $2g$. For $t\in \F_q$, let $A_t$ be the Jacobian of the
  hyperelliptic curve $C_t$ with affine equation
$$
C_t\,:\, y^2=f(x)(x-t).
$$
\par
Then we have
\begin{equation}\label{eq:siftingset}
|\{t\in\F_q\,\mid\, f(t)\not=0\text{ and }
A_t\text{ is not geometrically simple}\}|
\ll g^2q^{1-\gamma^{-1}}(\log q)
\end{equation}
where $\gamma=4g^2+2g+4$ and the implied constant is \emph{absolute}.
\end{prop}

\begin{proof}
  Fix a prime number $\ell\not=p$. For $t\in \F_q$, we let $P_t$
  denote the numerator of the zeta function of $C_t$, which is the
  integral polynomial of degree $2g$ given by
$$
P_t=\det(1-TF\mid H^1(A_t,\Z_{\ell})),
$$
where $H^1(A_t,\Z_{\ell})\simeq H^1(C_t,\Z_{\ell})$ is the first
\'etale cohomology group of $A_t$ or $C_t$ (this is the ``spectral
interpretation'' of the zeros of the zeta function of $C_t$).
\par
Let $G_t$ be the Galois group of the splitting field of $P_t$. We
write $W$ for the group which is the ``generic'' value of $G_t$,
namely the Weyl group of the symplectic group $Sp(2g)$, or more
concretely, the group of order $2^gg!$ consisting of signed permutation matrices in $GL(n,\Z)$.
\par
From the application of the sieve for Frobenius in~\cite[Remark after
Th. 8.13]{kow:lscup}, it follows that
$$
|\{t\in\F_q\,\mid\, f(t)\not=0 \text{ and }
G_t\not\simeq W\}|\ll
g^2q^{1-\gamma^{-1}}(\log q),
$$
where $\gamma=4g^2+2g+4$ and the implied constant is \emph{absolute}
(the earlier result in~\cite[Th. 6.2]{kow:ls} has $\gamma=4g^2+3g+5$
instead, which is virtually indistinguishable; it also misses the
$g^2$ factor, due to a slip in the final step of the estimate).
\par
Precisely, this result trivially implies~(\ref{eq:siftingset}) if
``geometrically simple'' is replaced by ``simple'', since an isogeny
(over $\F_q$) of the type
\begin{equation}\label{eq-factor}
A_t\simeq A_1\times A_2
\end{equation}
with $\dim A_1$, $\dim A_2\geq 1$, implies that
\begin{equation}\label{eq-factor2}
P_t=\det(1-TF\mid H^1(A_1,\Z_{\ell}))\det(1-TF\mid H^1(A_2,\Z_{\ell})),
\end{equation}
where both factors are integral polynomials of degree $\geq 1$, which
can certainly not occur if $P_t$ has Galois group $W$.
\par
To claim the result stated in the geometric context, one must exclude
factorizations as above which hold only over a finite extension of
$\F_q$. For fixed $g$, one can adapt straightforwardly the
corresponding qualitative argument of Chavdarov~\cite[Th. 2.1, Lemma
5.3]{chavdarov}. The dependency on $g$ might be worse than what we
claim when applying this directly, but for $g\geq 5$ (at least), one
can use instead the following elementary argument exploiting the size
of the Galois group.  First, one can show (see~\cite[Prop. 2.4,
(2)]{kow:indep}) that $G_t\simeq W$ and $g\geq 5$ imply that the only
multiplicative relations between zeros of $P_t$ must follow from the
Riemann Hypothesis, i.e., if $(\alpha_1,\ldots, \alpha_{2g})$ are the
inverse roots of $P_t$, we have $\Q\otimes_{\Z} R=T$, where
\begin{gather*}
R=\{(n_i)\in\Z^{2g}\,\mid\, \prod_{i}{\alpha_i^{n_i}}=1\},\\
T=\{(m_i)\in\Q^{2g}\,\mid\, \sum_j{m_j}=0,\text{ and }
m_i=m_j\text{ if } \alpha_i=\bar{\alpha}_j\}.
\end{gather*}
\par
Now if~(\ref{eq-factor}) holds over $\F_{q^m}$, $m\geq 1$, it is easy
to see that there must be a relation $\alpha_j^m=\alpha_k^m$ with
$j\not=k$, and this corresponds to a relation $(n_i)\in R$ with
$n_i=0$ except $n_j=n_k=m$, which is incompatible with the definition
of $T$.
\end{proof}

\begin{rem}
  The uniformity in $g$ is a nice additional feature of the sieve
  method, but it is not necessarily crucial here; the uniformity in
  terms of the characteristic of $\F_q$ is what matters for the later
  use of this proposition.
\par
It is worth noting one common feature of the geometric and analytic
approaches here: the proof of Proposition~\ref{pr:frobsieve} depends
crucially on the same result of J-K. Yu (reproved
in~\cite{hall:monodromy}) concerning the monodromy modulo $\ell$ of
our hyperelliptic families, over finite fields.
\end{rem}

\begin{thm}\label{th:analytic}
  Let $k/\Q$ be a number field, $g\geq 1$ an integer and $f\in k[X]$ a
  squarefree polynomial of degree $2g$. For $t\in k$, not a zero of
  $f$, let $A_t$ be the Jacobian of the hyperelliptic curve with
  affine equation
$$
y^2=f(x)(x-t).
$$
\par
For $B\geq 1$, let
$$
S(B)=\{t\in k\,\mid\, H(t)\leq B\text{ and }
A_t\text{ is not geometrically simple}\}.
$$
\par
Then there exists an absolute constant $D\geq 0$ such that, for $B\geq
2$, we have
$$
|S(B)|\ll (\log 2^{[k:\Q]}B^2)^{\gamma-1}(g^2D\log \log 2^{[k:\Q]}B)^{\gamma},
$$
with $\gamma=4g^2+2g+4$, where the implied constant depends only on
$k$.
\end{thm}

\begin{proof}
  The basic observation is that, if $t\in S(B)$ then for any prime
  ideal $\ideal{p}$, $t\mods{p}\in \P^1(\F_{\ideal{p}})$ is either a
  zero of $f$ modulo $\ideal{p}$, or $\infty$, or else ($f(t)$ being
  non-zero modulo $\ideal{p}$ so that $A_t$ has good reduction modulo
  $\ideal{p}$, and its fiber over $\ideal{p}$ then being not
  geometrically simple), $t\mods{\ideal{p}}$ lies in the set
  $\Omega_{\ideal{p}}$ defined by~(\ref{eq:siftingset}) for $f$
  relative to $q=N\ideal{p}$.
\par
Hence the image of $S(B)$ modulo $\ideal{p}$ has cardinality
$\nu(\ideal{p})$ with
$$
\nu(\ideal{p})
\leq 2g+1+|\Omega_{\ideal{p}}|
\ll g^2 (N\ideal{p})^{1-\gamma^{-1}}(\log N\ideal{p}),
$$
where the implied constant is absolute by
Proposition~\ref{pr:frobsieve}. Thus Corollary~\ref{cor:2} directly
implies the result.
\end{proof}

\begin{rem}
  In an extremely narrow range, the large sieve (as used originally
  in~\cite{kow:ls}) is better than the larger sieve.  Indeed, as
  discussed with many examples in~\cite{croot-elsholtz}, the original
  larger sieve is better when the number of permitted residue classes
  (i.e., the size of $\Omega_{\ideal{p}}$, in our case) is smaller
  than half of $N\ideal{p}$ (this is not quite true anymore in our
  inequality because of the term $\log 2^{[k:\Q]}B^2$ in the
  denominator).  Proposition~\ref{pr:frobsieve} clearly shows that we
  can not prove this\footnote{\ It may be true, for all we know.}
  unless $N\ideal{p}$ is (roughly) larger than $\delta^{-1/\delta}$
  (with $\delta\asymp g^2$). But the bound in
  Proposition~\ref{pr:frobsieve} also becomes trivial for $g$ not much
  beyond this point, so the range of applicability where the large
  sieve would be the best is very small.
\end{rem}

\section*{Appendix: survey of abelian varieties for analytic number
  theorists}

While the basic information about abelian varieties that we use will
certainly be well-known to readers more familiar with the methods of
Sections~\ref{sec:geoI} and~\ref{sec:geoII}, this is less likely to be
the case for readers whose interests lie more in the direction of
analytic number theory and sieves. In order to motivate the basic
problem for these readers, we summarize here briefly some background
information, which we hope will suffice to make accessible the
contents of Section~\ref{sec:analytic} for such readers.
\par
The simplest case of abelian varieties is that of elliptic curves;
although our basic question of geometric simplicity is not of interest
in this setting (any elliptic curve is geometrically simple), a basic
knowledge of elliptic curves can help motivate and understand the
general theory. We refer for this to Silverman's
book~\cite{silverman}, and to the summary in~\cite[\S 11.10]{ik},
which may also be helpful.
\par
Let $k$ be a number field (for instance, $k=\Q$). An abelian variety
$A$ defined over $k$ is, first of all, a {\em proper variety} over
$k$; that is, we may think of $A$ as a subset of projective space over
$k$ cut out by some set of homogeneous equations in the coordinates.
(In practice, though, one almost never writes down these equations!)
What makes $A$ an abelian variety is the presence of a {\em group
  law}: a map from $A \times A$ to $A$ which is given by polynomials
in the coordinates, and satisfies the usual group axioms --
associativity, presence of an inverse, and so on.  (One might compare
$A$ with the more familiar example of $\SL_n/k$, which is also
determined as a subset of $k^{n^2}$ by a set of equations, and which
also has a group operation which is polynomial in the matrix entries.
The difference is that $A$ is cut out by equations in projective
space, while $\SL_n$ is cut out by equations in the affine space
$k^{n^2}$.)
\par
Since $k$ is contained in $\C$, we can ask not only about the group of
solutions over $k$ to the defining equations of $A$, but about the set
of complex solutions, denoted $A(\C)$.  Write $g$ for the dimension of
$A$.  It is known that $A$ is necessarily isomorphic to $\C^g/\Lambda$
for some lattice $\Lambda\simeq \Z^{2g}\subset \C^g$; in the
$1$-dimensional case $g=1$, $A$ is an elliptic curve over $k$.
\par
In particular, it follows that the subgroup $A[n]$ of elements of
order dividing $n$ in $A$, for any integer $n\geq 1$, is isomorphic to
$(\Z/n\Z)^{2g}$, and moreover the fact that $A$ is defined over $k$
easily implies that the coordinates of elements in $A[n]$ are
algebraic numbers, which all together generate a finite Galois
extension $k(A[n])$ of $k$.
\par
Algebraic curves provide a natural source of abelian varieties via the
construction of the {\em Jacobian}, which over $\C$ goes back to
Jacobi, and over $k$ to Weil.  To each non-singular algebraic curve
$C/k$ of genus $g$, one can attach a natural abelian variety $J(C)$
over $k$ of dimension $g$.  One nice feature of Jacobians is that they
are {\em principally polarized}: this is a kind of self-duality which
imposes on $J(C)[n]$ a natural perfect pairing
$$
J(C)[n]\times J(C)[n]\ra \text{\boldmath$\mu$}_n\simeq \Z/n\Z
$$
where $\text{\boldmath$\mu$}_n$ denotes the group of $n$-th roots of
unity.
\par
In fact, the action of $\Gal(\bar{k}/k)$ on the coordinates of
$k(A[n])$ is not merely linear, but compatible with the symplectic
pairing above; thus it provides a representation
$$
\Gal(\bar{k}/k)\ra \Aut(A[n])\simeq GSp(2g,\Z/n\Z).
$$
\par
The primary examples of abelian varieties treated in this paper are
Jacobians of curves; in any event, all the abelian varieties we
consider are for simplicity assumed to be principally polarized.
\par
The most delicate issue for Section~\ref{sec:analytic} is that of
reductions of an abelian variety modulo prime ideals of
$\Z_k$. Suffice it to say here that this can be defined for all but
finitely many prime ideals of $k$ (the ``primes of bad reduction''),
and that if concrete equations for $A$ are given so that, modulo
$\ideal{p}$, the resulting equations still define a smooth algebraic
variety, then the reduction coincides pretty much with the na\"{\i}ve
notion of looking at solutions of the equations with coefficients in
extensions of the residue field $\Z_k/\ideal{p}$.
\par
Now our basic problem takes root in the following definition: an
abelian variety $A/k$ is \emph{simple} if and only if there is no
nontrivial abelian variety $B$ over $k$ which is a subvariety of $A$,
except $A$ itself.  It is {\em geometrically simple} if it remains
simple even when considered as an abelian variety over $\C$.
\par
Implicit in the notion of geometric simplicity is that, for most
lattices $\Lambda \in \C^g$, the quotient $\C^g / \Lambda$ is not an
abelian variety.  It is merely a complex torus; the condition that it
embeds as an algebraic subvariety of projective space imposes very
strong restrictions on $\Lambda$ (originally described by Riemann.)
In particular, if $\C^g / \Lambda$ is an abelian variety, it is not
usually possible to find a subspace $V \in \C^g$ such that $\Lambda
\cap V$ is a lattice in $V$ and $V/(\Lambda \cap V)$ is an abelian
variety.  In other words, abelian varieties over $\C$ are
``typically'' simple.
\par
Now the question considered in this paper is essentially the
following: we form a family, parameterized by elements in $k$, of
curves; then we have an associated family of Jacobian varieties, and
we ask: \emph{how frequent is it that those abelian varieties are not
  geometrically simple}? 
\par
The basic approach in Section~\ref{sec:analytic} is founded on the
following fact: if an abelian variety $A/k$ is not geometrically
simple, then its reduction modulo a prime ideal $\ideal{p}$ has the
same property (which is intuitive enough). Moreover, a result going
back in principle to Poincar\'e shows that a non-trivial subvariety
$B\subset A$ is ``essentially'' a direct factor, i.e., we have
$$
A\simeq B\times C
$$
for some other abelian subvariety $C$, up to finite groups (``up to
isogeny''). This is the property~(\ref{eq-factor}) which leads to the
factorization~(\ref{eq-factor2}) which we use to control the occurence
of non-geometrically simple varieties.

\end{document}